\documentclass[a4paper,10pt]{amsart}
\usepackage[english]{babel}
\usepackage[OT2,T1]{fontenc}
\usepackage[utf8x]{inputenc}
\usepackage{lmodern}
\usepackage{microtype}
\usepackage{amsmath,amssymb,amsfonts,amsthm}
\usepackage{latexsym,mathrsfs}
\usepackage{graphicx}
\usepackage[all]{xy}
\input xypic
\usepackage{color}
\usepackage{hyperref}
\hypersetup{colorlinks=true,linkcolor=red,citecolor=blue}
\theoremstyle{plain}
\newtheorem{theorem}[subsection]{{\bf Theorem}}
\newtheorem*{theorem*}{{\bf Theorem}}
\newtheorem{corollary}[subsection]{{\bf Corollary}}
\newtheorem*{corollary*}{{\bf Corollary}}
\newtheorem{proposition}[subsection]{{\bf Proposition}}
\newtheorem{lemma}[subsection]{{\bf Lemma}}

\theoremstyle{definition}

\theoremstyle{remark}
\newtheorem{remark}[subsection]{{\it Remark}}
\newtheorem*{remark*}{{\it Remark}}
\newtheorem{example}[subsection]{{\it Example}}
\numberwithin{equation}{subsection}
\DeclareMathOperator{\im}{im}

\DeclareMathOperator{\Ext}{Ext}

\DeclareMathOperator{\HH}{H}
\DeclareMathOperator{\K}{K}
\DeclareMathOperator{\B}{B}
\DeclareMathOperator{\Z}{Z}
\DeclareMathOperator{\Bt}{B}

\DeclareMathOperator{\M}{M}

\DeclareMathOperator{\dd}{d}
\DeclareMathOperator{\Aut}{Aut}
\DeclareMathOperator{\Hom}{Hom}
\DeclareMathOperator{\GL}{GL}

\DeclareMathOperator{\cp}{cp}

\DeclareMathOperator{\UT}{UT}
\DeclareMathOperator{\rank}{rank}
\DeclareMathOperator{\rr}{r}
\DeclareSymbolFont{cyrletters}{OT2}{wncyr}{m}{n}
\DeclareMathSymbol{\Sha}{\mathalpha}{cyrletters}{"58}
\newcommand{\QZ}{\mathbb{Q}/\mathbb{Z}}
\newcommand{\ZCP}{\Z _{\rm CP}}
\newcommand{\CC}{\mathbb C}

\newcommand{\HHCP}{\HH_{\rm CP}}

\newcommand{\cwedge}{\curlywedge}

\newcommand{\FF}{\mathbb F}
\begin{document}
\baselineskip=14pt
\title[CP extensions]{Commutativity preserving extensions of groups}
\author[U. Jezernik]{Urban Jezernik}
\address[Urban Jezernik]{
Institute of Mathematics, Physics, and Mechanics \\
Jadranska 19 \\
1000 Ljubljana \\
Slovenia}
\thanks{}
\email{urban.jezernik@imfm.si}
\author[P. Moravec]{Primo\v{z} Moravec}
\address[Primo\v{z} Moravec]{
Department of Mathematics \\
University of Ljubljana \\
Jadranska 21 \\
1000 Ljubljana \\
Slovenia}
\thanks{}
\email{primoz.moravec@fmf.uni-lj.si}
\thanks{}
\date{\today}
\begin{abstract}

\noindent In parallel to the classical theory of central extensions of groups,
we develop a version for extensions that preserve commutativity. It is shown
that the Bogomolov multiplier is a universal object parametrizing such
extensions of a given group. Maximal and minimal extensions are inspected,
and a connection with commuting probability is explored. Such considerations
produce bounds for the exponent and rank of the Bogomolov multiplier.

\end{abstract}
\maketitle
\section{Introduction}
\label{s:intro}

\noindent
Noether's problem \cite{Noe16} is one of the fundamental problems of invariant theory, and asks  
as to whether the field of $Q$-invariant functions $\mathbb{C}(V)^Q$ is purely transcendental over $\mathbb{C}$, where $Q$ is a given finite group. Artin and Mumford \cite{Art72} introduced an obstruction $\HH^2_{\rm nr}(\mathbb{C}(V)^Q,\QZ )$ to this problem, called the {\it unramified Brauer group} of the field extension $\mathbb{C}(V)^Q/\mathbb{C}$. In his seminal work, Bogomolov \cite{Bog88} proved that $\HH^2_{\rm nr}(\mathbb{C}(V)^Q,\QZ )$ is canonically isomorphic to the intersection of the kernels of restriction maps $\HH^2(Q,\QZ)\to \HH^2(A,\QZ)$, where $A$ runs through all abelian subgroups of $Q$. 
A simplified description of $\HH^2_{\rm nr}(\mathbb{C}(V)^Q,\QZ )$ was found in \cite{Mor12} by considering its dual $\B_0(Q)$. Following Kunyavski\u\i\ \cite{Kun08}, we call the latter group the {\it Bogomolov multiplier} of $Q$. The description of $\B_0$ is combinatorial and enables efficient explicit calculations. Furthermore, it relates Bogomolov multipliers to the commuting probability of a group \cite{Jez15}, and shows that $\B_0$ plays a role in describing the so-called {\it commutativity preserving central extensions} of groups, which are in close relationship with some problems in K-theory \cite{Mor12}.

In this paper we develop a theory of commutativity preserving group extensions with abelian kernel. Specifically, let $Q$ be a group and $N$ a $Q$-module. Denote by $e=(\chi , G,\pi )$ the extension
\begin{equation*}
\xymatrix{ 1\ar[r] & N\ar[r]^\chi & G\ar[r]^\pi & Q\ar[r] & 1}
\end{equation*}
of $N$ by $Q$.
Following \cite{Mor12}, we say that $e$ is a {\it CP extension} if commuting
pairs of elements of $Q$ have commuting lifts in $G$.
 In the first part of the paper we define a subgroup $\HHCP^2(Q,N)$ of the second cohomology group $\HH^2(Q,N)$ that classifies CP extensions of $N$ by $Q$ up to equivalence. Then we focus on central CP extensions. We prove a variant of the Universal Coefficient Theorem by showing that, given a trivial $Q$-module $N$, there is a split exact sequence
\begin{equation*}
\xymatrix{0\ar[r] & \Ext (Q^{\rm ab},N)\ar[r] & 
	\HHCP^2(Q,N)\ar[r] & \Hom(\Bt_0(Q),N)\ar[r] & 0}.
\end{equation*}
In addition to that, we provide several characterizations of central CP extensions, and prove that these are closed under isoclinism of extensions. Subsequently, we show that the isoclinism classes of central CP extensions with a given factor group $Q$ are in bijective correspondence with the orbits of the action of $\Aut Q$ upon the subgroups of $\B_0(Q)$.

In what follows, we consider stem central CP extensions of $N$ by $Q$, where $|N|=|\B_0(Q)|$. We call such extensions {\it CP covers} of $Q$. These are analogs of the usual covers in the theory of Schur multipliers. We show that every finite group has a CP cover, and that all CP covers of isoclinic groups are isoclinic.
Further, we show how CP covers are, in a suitable sense, precisely the maximal central CP extensions of $Q$.
In the succeeding section we then also consider minimal central CP extensions, i.e., those whose kernel is a cyclic group of prime order. Such extensions are parametrized by $\HHCP^2(Q,\mathbb{F}_p)$. The main result in this direction is that this group is an elementary abelian $p$-group of rank $\dd(Q)+\dd(\B_0(Q))$.

Applying the theory of CP covers, we derive some bounds for the order, rank, and exponent of the Bogomolov multiplier of a given finite group $Q$. We obtain bounds for $\B_0(Q)$ that correspond to those for Schur multipliers obtained by Jones and Wiegold \cite{Jon73} and Jones \cite{Jon74}. On the other hand, a special feature of $\B_0$ is that it is closely related to the commuting probability of the group in question, that is, the probability that two randomly chosen elements of the group commute. This was already explored in \cite{Jez15} where we proved that if the commuting probability of $Q$ is strictly greater than $1/4$, then $\B_0(Q)$ is trivial. Here we prove that if the commuting probability of a finite group $Q$ is strictly greater than a fixed $\epsilon>0$, then the order of $\B_0(Q)$ can be bounded in terms of $\epsilon$ and the maximum of minimal numbers of generators of Sylow subgroups of $Q$. Furthermore, we show that $\exp B_0(Q)$ can be bounded in terms of $\epsilon$ only.

\section{CP extensions}
\label{s:acp}

\noindent
The purpose of this section is to establish a cohomological object that encodes all the information on CP extensions up to equivalence. We refer to \cite{Bro82} for an account on the theory of group extensions.

\begin{lemma}
\label{l:equiv}
The class of CP extensions is closed under equivalence of extensions.
\end{lemma}

\begin{proof}
Let 
\begin{equation*}
\xymatrix{ 
0\ar[r] & N\ar[r]^{\mu _1}\ar@{=}[d] & G_1\ar[d]^\theta\ar[r]^{\epsilon _1} & 
	Q\ar@{=}[d]\ar[r] & 1\\
0\ar[r] & N\ar[r]^{\mu _2} & G_2\ar[r]^{\epsilon _2} & Q\ar[r] & 1
}
\end{equation*}
be equivalent extensions with abelian kernel. Suppose that $G_1$ is a CP
extension of $N$ by $Q$. Choose $x_1,x_2\in Q$ with $[x_1,x_2]=1$. Then there
exist $e_1,e_2\in G_1$ such that $[e_1,e_2]=1$ and $\epsilon _1(e_i)=x_i$,
$i=1,2$. Take $\bar{e}_i=\theta (e_i)$. Then $[\bar{e}_1,\bar{e}_2]=1$ and
$\epsilon _2(\bar{e}_i)=x_i$. This proves that $G_2$ is a CP extension of $N$ by
$Q$.
\end{proof}

As in the classical setting of group extensions, CP extensions can be interpreted
in a cohomological manner.
Let $Q$ and $S$ be groups, and suppose that $Q$ acts on $S$ via 
$(x,y)\mapsto {}^xy$, where $x\in Q$ and $y\in S$. A map $\partial : Q\to S$ is
a {\it derivation} (or  {\it 1-cocycle}) {\it from $Q$ to $S$} if $\partial
(xy)={}^x\partial (y)\partial (x)$ for all $x,y\in Q$. Let $N$ be a $Q$-module
and fix $a\in N$. The map $\partial _a:Q\to N$, given by $\partial _a(g)=ga-a$,
is a derivation. It is called an {\it inner derivation}.

A cocycle $\omega\in \Z^2(Q,N)$ is said to be a {\it CP cocycle} if for all
commuting pairs $x_1,x_2\in Q$ there exist $a_1,a_2\in N$ such that
\begin{equation}
\label{eq:cpcocyc}
\omega(x_1,x_2)-\omega(x_2,x_1)=\partial _{a_1}(x_1)+\partial_{a_2}(x_2).
\end{equation}
Denote by $\ZCP^2(Q,N)$ the set of all CP cocycles in $\Z ^2(Q,N)$.

\begin{proposition}
\label{p:subgrp}
$\ZCP ^{2}(Q,N)$ is a subgroup of $\Z ^2(Q,N)$ containing $\B ^2(Q,N)$.
\end{proposition}

\begin{proof}
It is clear that $\ZCP ^{2}(Q,N)$ is a subgroup of $\Z ^2(Q,N)$. Now let
$\beta\in\B ^2(Q,N)$. Then there exists a function $\phi :Q\to N$ such that
\[
\beta (x_1,x_2) = x_1\phi (x_2)-\phi (x_1x_2)+\phi (x_1)
\]
for all $x_1,x_2\in Q$. Suppose that these two elements commute. Then
$\beta(x_1,x_2)-\beta (x_2,x_1)=\partial_{\phi (x_2)}(x_1)+\partial_{-\phi(x_1)}(x_2)$,
hence $\beta\in\ZCP^2(Q,N)$.
\end{proof}

Now define $\HHCP^2(Q,N) = \ZCP^2(Q,N) / \B^2(Q,N)$. This is a subgroup of the
ordinary cohomology group $\HH^2(Q,N)$.

\begin{example}
\label{e:ext}
Let $Q$ be an abelian group and $N$ a trivial $Q$-module. Then $\HHCP ^2(Q,N)$
coincides with $\Ext (Q,N)$.
\end{example}

\begin{proposition}
\label{p:classext}
Let $N$ be a $Q$-module. Then the equivalence classes of CP extensions of $N$ by
$Q$ are in bijective correspondence with the elements of $\HHCP ^2(Q,N)$.
\end{proposition}

\begin{proof}
Let $e=(\chi, G,\pi )$ be an extension of $N$ by $Q$. Let $\omega :Q\times Q\to
N$ be a corresponding 2-cocycle. Then $e$ is equivalent to the extension
\begin{equation*}
\xymatrix{ 1\ar[r] & N\ar[r] & Q[\omega]\ar[r]^\epsilon & Q\ar[r] & 1},
\end{equation*}
where $Q[\omega]$ is, as a set, equal to $N\times Q$, and the operation is given
by $(a,x)(b,y)=(a+xb+\omega (x,y),xy)$, and $\epsilon (a,x)=x$. By Lemma
\ref{l:equiv} it suffices to show that the latter extension is CP if and only if
$\omega\in\ZCP^2(Q,N)$. Let $x,y\in Q$ commute and let $(a,x)$ and $(b,y)$ be
lifts of $x$ and $y$ in $Q[\omega ]$. Then $(a,x)$ and $(b,y)$ commute if and
only if $\omega (x,y)-\omega (y,x)=(y-1)a-(x-1)b=\partial _a(y)+\partial
_{-b}(x)$. Thus the existence of commuting lifts of $x$ and $y$ is equivalent to
$\omega\in \ZCP^2(Q,N)$.
\end{proof}

We now give some examples.

\begin{example}

Let $Q$ be a group in which for every commuting pair $x,y$ the subgroup $\langle
x, y \rangle$ is cyclic. This is equivalent to $Q$ having all abelian subgroups
cyclic. In the case of finite groups, it is known \cite[Theorem VI.9.5]{Bro82} that such groups are precisely
the groups with periodic cohomology, and this further amounts
to $Q$ having cyclic Sylow $p$-subgroups for $p$ odd, and cyclic or quaternion Sylow
$p$-subgroups for $p = 2$. Infinite groups with this property include free products
of cyclic groups, cf. \cite{Kar59}. Given such a group $Q$, it is clear that every commuting pair
of elements in $Q$ has a commuting lift. Thus every extension of $Q$ is CP,
and so $\HH^2(Q, N) = \HHCP^2(Q,N)$ for any $Q$-module $N$.
\end{example}

\begin{example}

Taking the simplest case $Q = C_p$ in the previous example, we see that every
extension of a group by $C_p$ is CP. Thus in particular, every finite $p$-group can be
viewed as being composed from sequence of CP extensions.

\end{example}

\begin{example}

There are many examples of extensions that are not CP. One may simply take as
$G$ a group of nilpotency class $2$ and factor by a subgroup generated by a non-trivial
commutator. In fact, in the case when the extension is central, it is more
difficult to find examples of extensions that are CP. We will focus on
inspecting central CP extensions in the following section. Consider now only
extensions that are not central. Some small examples of extensions which fail to
be CP are easily produced by taking a non-trivial action of a non-cyclic abelian
group on an elementary abelian group. We give a concrete example. Take $Q =
\langle x_1 \rangle \times \langle x_2 \rangle$ to be an elementary abelian
$p$-group of rank $2$, and let it act on $N = \langle a_1 \rangle \times \langle
a_2 \rangle \times \langle a_3 \rangle$, an elementary abelian $p$-group of rank
$3$, via the following rules:  \[ a_1^{x_1} = a_1, \; a_2^{x_1} = a_2, \;
a_3^{x_1} = a_3, \; a_1^{x_2} = a_2, \; a_2^{x_2} = a_1, \; a_3^{x_2} = a_3. \]
Thus $N$ is a $Q$-module. Now construct an extension $G$ corresponding to this
action by specifying $x_2^{x_1} = x_2 a_3$. This extension is not CP because the
commuting pair $x_1, x_2$ in $Q$ does not have a commuting lift in $G$.

\end{example}

\section{Central CP extensions}
\label{s:ccp}

From now on, we focus on a special type of CP extensions -- those with central
kernel. In terms of the cohomological interpretation, these correspond to the
case when the relevant module is trivial.

The fundamental result here is a CP version of the Universal Coefficient
Theorem. In other words, there exists a universal cohomological object that parametrizes
all central CP extensions. We show below that this object is the Bogomolov multiplier. 
Let us first recall its definition in more detail \cite{Mor12}. Given a group $Q$, let $Q\wedge Q$ be the group generated by the symbols $x\wedge y$, where $x,y\in Q$, subject to the following relations:
\begin{equation}
\label{eq:ext}
xy\wedge z = (x^y\wedge z^y)(y\wedge z), \quad 
x\wedge yz = (x\wedge z)(x^z\wedge y^z), \quad
x\wedge x = 1,
\end{equation}
where $x,y,z\in Q$. The group $Q\wedge Q$ is said to be the {\it non-abelian
exterior square} of $Q$, defined by Miller \cite{Mil52}. There is a surjective
homomorphism $Q\wedge Q\to [Q,Q]$ given by $x\wedge y\mapsto [x,y]$. Miller
\cite{Mil52} showed that the kernel $\M (Q)$ of this map is naturally isomorphic
to the Schur multiplier $\HH_2(Q,\mathbb{Z})$ of $Q$. Finally, denote
$\B_0(Q)=\M (Q)/\M _0(Q)$, where $\M _0(Q)=\langle x\wedge y\mid x,y\in Q,\,
[x,y]=1\rangle$; this is the {\em Bogomolov multiplier}.
One
can therefore consider $\B_0(Q)$ as the kernel of the induced commutator map
from the {\it non-abelian curly exterior square} $Q \curlywedge Q=(Q\wedge
Q)/\M_0(Q)$ to $[Q,Q]$.
It is shown in \cite{Mor12} that $\HH^2_{\rm
nr}(\mathbb{C}(V)^Q,\QZ )$ is naturally isomorphic to $\Hom (\B_0(Q),\QZ)$. 

\begin{theorem}
\label{t:uct2}
Let $N$ be a trivial $Q$-module. Then there is a split exact sequence
\begin{equation}
\label{eq:buct}
\xymatrix{0\ar[r] & \Ext (Q^{\rm ab},N)\ar[r]^\psi & 
	\HHCP^2(Q,N)\ar[r]^{\tilde{\varphi}} & \Hom(\Bt_0(Q),N)\ar[r] & 0},
\end{equation}
where the maps $\psi$ and $\tilde{\varphi}$ are induced by the Universal
Coefficient Theorem.
\end{theorem}

\begin{proof}
By the Universal Coefficient Theorem, we have a split exact sequence
\begin{equation}
\label{eq:uct}
\xymatrix{0\ar[r] & \Ext (Q^{\rm ab},N)\ar[r]^\psi & 
	\HH^2(Q,N)\ar[r]^{\varphi} & \Hom(\M(Q),N)\ar[r] & 0}.
\end{equation}
Let $[\omega]$ belong to $\Ext (Q^{\rm ab},N)$. Then \cite{Bey82} the map $\psi$
can be described as $\psi([\omega])=[\omega\circ ({\rm ab}\times {\rm ab})]$,
where ${\rm ab}:Q\to Q^{\rm ab}$. If $x,y\in Q$ commute, then
$\psi([\omega])(x,y)=\omega (x[Q,Q],y[Q,Q])=\omega (y[Q,Q],x[Q,Q])=\psi([\omega])(y,x)$,
therefore $\psi$ maps the group $\Ext (Q^{\rm ab},N)$ into $\HHCP^2(Q,N)$. The map
$\varphi$ can be described as follows. Suppose that $[\omega]\in\HH^2(Q,N)$
represents a central extension
\begin{equation}
\label{eq:cent}
\xymatrix{0\ar[r] & N\ar[r] & \tilde{Q}\ar[r]^\pi & Q\ar[r] & 1}.
\end{equation}
Let $z=\prod_i(x_i\wedge y_i)\in\M (Q)$, that is, $\prod_i[x_i,y_i]=1$. Choose
$\tilde{x}_i,\tilde{y}_i\in\tilde{Q}$ such that $\pi(\tilde{x}_i)=x_i$ and $\pi
(\tilde{y}_i)=y_i$. Define $\tilde{z}=\prod_i[\tilde{x}_i,\tilde{y}_i]$. Clearly
$\tilde{z}\in N$, and it can be verified that the map $\varphi$ is well defined
by the rule $\varphi ([\omega ])=(z\mapsto\tilde{z})$.

Suppose now that $[\omega]$ belongs to $\HHCP^2(Q,N)$. Let $z$ belong
to $\M_0(Q)$. Then $z$ can be written as $z=\prod_i(x_i\wedge y_i)$, where
$[x_i,y_i]=1$ for all $i$. Since the extension \eqref{eq:cent} is a central CP
extension, we can choose commuting lifts $(\tilde{x}_i,\tilde{y}_i)$ of the
commuting pairs $(x_i,y_i)$. By the above definition, $\tilde{z}=0$, hence
$\varphi$ is trivial when restricted to $\M_0(Q)$. Thus $\varphi$ induces an
epimorphism $\tilde{\varphi}:\HHCP^2(Q,N)\to \Hom (\Bt_0(Q),N)$ such that the
following diagram commutes:
\begin{equation*}
\xymatrix{\HH^2(Q,N)\ar[r]^\varphi & \Hom(\M(Q),N)\\
\HHCP^2(Q,N)\ar[r]^{\tilde{\varphi}}\ar[u]^\iota & \Hom (\Bt_0(Q),A)\ar[u]_{\rho^*}
}
\end{equation*}

Here the map $\rho^*$ is induced by the canonical epimorphism 
$\rho :\M(Q)\to\Bt_0(Q)$. Therefore it follows that
$\ker\tilde{\varphi}=\ker\varphi|_{\im\iota}=\im\psi$. This shows that the sequence \eqref{eq:buct} is exact. Furthermore, the splitting of the sequence \eqref{eq:uct} yields that the sequence \eqref{eq:buct} is also split. This proves the result.
\end{proof}

We offer a sample application the above theorem.  Recall that Schur's theory of
covering groups originally arised in the context of projective representations,
cf. \cite{Sch07}. Schur showed that there is a natural correspondence between
the elements of $\HH^2(Q,\CC^\times)$ and projective representations of $Q$. To
every projective representation $\rho \colon Q \to \GL(V)$ one can associate a
cocycle $\alpha \in \Z^2(Q, \CC^\times)$ via the rule $\rho(x) \rho(y) =
\alpha(x,y) \rho(xy)$ for every $x,y \in Q$. Projectively equivalent
representations induce cohomologous cocycles, and a cocycle is a coboundary if
and only if the representation is equivalent to a linear representation.
It is readily verified that CP extensions integrate well into this setting.

\begin{proposition}

Projective representations $\rho \colon Q \to \GL(V)$ with the property that
$[\rho(x_1), \rho(x_2)] = 1$ whenever $[x_1,x_2] = 1$ correspond to
cohomological classes of CP cocycles $\alpha \in \Z^2(Q, \CC^\times)$, i.e.,
elements of $\B_0(Q)$.

\end{proposition}

In particular, if $\B_0(Q)$ is trivial, every projective representation of $Q$
which preserves commutativity is similar to a linear representation. Such maps
have been studied in detail in other algebraic structures, see \cite{Sem08} for a survey. 
A loose connection can be made along the following lines. Let $\rho \colon Q \to
S$ be a set-theoretical map from $Q$ to a group $S$ such that $\rho(1) = 1$ and
the induced map $\rho \colon Q \to S/Z(S)$ is a homomorphism. We may thus write
$\rho(x)\rho(y) = \alpha(x,y) \rho(xy)$ for some function $\alpha \colon Q
\times Q \to Z(S)$. In view of the associativity of multiplication, $\alpha$ is
in fact a $Z(S)$-valued $2$-cocycle. As above, such maps $\rho$ correspond to
elements of $\HHCP^2(Q, Z(S))$.

Next, we give a simple criterion for determining whether or not a given central
extension is CP. This result will later be used repeatedly.

\begin{proposition}
\label{p:cp}
Let
\begin{equation*}
e: \,\, \xymatrix{ 1\ar[r] & N\ar[r]^{\chi} & G\ar[r]^{\pi} & Q\ar[r] & 1}
\end{equation*}
be a central extension. Then $e$ is a CP extension if and only if $\chi (N)\cap \K (G)=1$.
\end{proposition}

\begin{proof}

Denote $M=\chi(N)$. Suppose that $M\cap \K(G)=1$. Choose $x,y\in Q$ with
$[x,y]=1$. We have $x=\pi (g)$ and $y=\pi (h)$ for some $g,h\in G$. Then $\pi
([g,h])=1$, hence $[g,h]\in M\cap \K(G)=1$. Thus $g$ and $h$ are commuting lifts
of $x$ and $y$, respectively.

Conversely, suppose that $e$ is a CP central extension. Choose $[g,h]\in M\cap
K(G)$. By assumption, there exists a commuting lift $(g_1,h_1)\in G\times G$ of
the commuting pair $(\pi (g),\pi (h))$. We can thus write $g_1=ga$, $h_1=hb$,
where $a,b\in M$. It follows that $1=[g_1,h_1]=[ga,hb]=[g,h]$, hence $M$ is a CP
subgroup of $G$.
\end{proof}

It is clear from the proof above that the implication from right to left also
holds for non-central extensions. In the general case, however, the equivalence
fails. For example, when $Q$ is a cyclic group and $G$ non-abelian, we certainly
have $\chi(N) \cap \K(G) = \K(G) > 1$, and the extension is CP.

We proceed with some further characterizations of central CP extensions. 
We say that a normal abelian subgroup $N$ of a group $G$ is a {\it CP subgroup}
of $G$ if the extension
\[
\xymatrix{ 1\ar[r] & N\ar[r] & G\ar[r] & G/N\ar[r] & 1}
\]
is a CP extension. In the case when $N$ is central in $G$, Proposition \ref{p:cp} implies that $N$ is a CP subgroup if and only $N\cap \K(G)=1$.
The following lemma will be needed.

\begin{lemma}
\label{l:b0cp}
Let $N$ be a central CP subgroup of $G$. Then the sequences
\[
0\longrightarrow \Bt_0(G)\longrightarrow \Bt_0(G/N)\longrightarrow N\cap G'\longrightarrow 0
\]
and
\[
N\otimes G^{\rm ab}\to \M_0(G)\to\M_0(G/N)\to 0
\]
are exact.
\end{lemma}

\begin{proof}
Let $G$ and $N$ be given via free presentations, that is, $G=F/R$ and $N=S/R$.
The fact that $N$ is a central CP subgroup of $G$ is then equivalent to $\langle
\K(F)\cap S\rangle \le R$.  This immediately implies that $\langle \K(F)\cap
S\rangle =\langle \K(F)\cap R\rangle$. With the above identifications and Hopf's formula for the Bogomolov multiplier \cite{Mor12} we have
that  $\Bt _0(G)=(F'\cap R)/\langle\K(F)\cap R\rangle$, $\Bt_0(G/N)=(F'\cap
S)/\langle \K(F)\cap S\rangle$, $\M_0(G)=\langle \K(F)\cap R\rangle/[F,R]$, and
$\M_0(G/N)=\langle \K(F)\cap S\rangle/[F,S]$. By \cite[p. 41]{Bey82} there is a Ganea map $N\otimes G^{\rm ab}\to \M(G)$ whose image 
can be identified with $[F,S]/[F,R]$.
As $[F,S]\le\langle \K(F)\cap R\rangle$, the Ganea map actually maps $N\otimes
G^{\rm ab}$ into $\M_0(G)$. The rest of the proof is now straightforward.
\end{proof}

\begin{proposition}
\label{p:cpm0}
Let $N$ be a central subgroup of a group $G$. The following are equivalent:
\begin{enumerate}
\item[(a)] $N$ is a CP subgroup of $G$,
\item[(b)] The canonical map $\M_0(G)\to\M_0(G/N)$ is surjective.
\item[(c)] The canonical map $\varphi:G\cwedge G\to G/N\cwedge G/N$ is an isomorphism.
\end{enumerate}
\end{proposition}

\begin{proof}
Let $G=F/R$ and $N=S/R$ be free presentations of $G$ and $N$. Then the image of
the map $\M_0(G)\to\M_0(G/N)$ can be identified with $\langle \K(F)\cap
R\rangle/[F,S]$. Thus the above map is surjective if and only if $\langle
\K(F)\cap R\rangle =\langle \K(F)\cap S\rangle$. In particular, $\langle
\K(F)\cap S\rangle \le R$, therefore $N$ is a CP subgroup of $G$.
This, together with Lemma \ref{l:b0cp}, shows that (a) and (b) are equivalent. Furthermore,
from \cite{Mor12} it follows that 
$\ker\varphi=\langle x\cwedge y\mid [x,y]\in N\rangle$. Hence $\varphi$ is
injective if and only if $\K (G)\cap N=1$, hence (a) and (c) are equivalent.
\end{proof}

We now discuss comparing different extensions. Let
\begin{equation*}
e_1: \,\, \xymatrix{ 1\ar[r] & N_1\ar[r]^{\chi _1} & G_1\ar[r]^{\pi _1} & Q_1\ar[r] & 1}
\end{equation*}
and 
\begin{equation*}
e_2: \,\, \xymatrix{ 1\ar[r] & N_2\ar[r]^{\chi _2} & G_2\ar[r]^{\pi _2} & Q_2\ar[r] & 1}
\end{equation*}

be central extensions. Following \cite{Bey82}, we say that $e_1$ and $e_2$ are
{\it isoclinic}, if there exist isomorphisms $\eta :Q_1\to Q_2$ and
$\xi :G_1'\to G_2'$ such that the diagram
\[
\xymatrix{
Q_1\times Q_1 \ar[r]^{c_1} \ar[d]^{\eta\times\eta} & G_1' \ar[d]^{\xi}\\
Q_2\times Q_2 \ar[r]^{c_2} & G_2'
}
\]
commutes, where the maps $c_i$, $i=1,2$, are defined by the rules 
$c_i(\pi _i(x),\pi _i(y))=[x,y]$. Note that these are well defined, since the
extensions are central.
\begin{proposition}
\label{p:isoclin}
Let $e_1$ and $e_2$ be isoclinic central extensions. If $e_1$ is
a CP extension, then so is $e_2$.
\end{proposition}

\begin{proof}
We use the same notations as above. Choose a commuting pair $(x_2,y_2)$ of
elements of  $Q_2$. Denote $x_2=\eta (x_1)$ and $y_2=\eta (y_1)$ where
$x_1,y_1\in Q_1$. Clearly $[x_1,y_1]=1$. As $e_1$ is a CP central extension, we
can choose commuting lifts $g_1,h_1\in G_1$ of $x_1$ and $y_1$, respectively. We
can write $x_2=\pi _2(g_2)$ and $y_2=\pi _2(h_2)$ for some $g_2,h_2\in G_2$. By
definition, $1=\xi ([g_1,h_1])=[g_2,h_2]$, hence $g_2$ and $h_2$ are commuting
lifts in $G_2$ of $x_2$ and $y_2$, respectively.
\end{proof}

We now show how CP extensions up to isoclinism of a given group can be
obtained from an action of its Bogomolov multiplier.

\begin{lemma}[\cite{Mor12}]
\label{l:B0exact}
Let $N$ be a normal subgroup of a group $G$. Then the sequence of groups
\[
\Bt_0(G) \longrightarrow
\Bt_0(G/N) \longrightarrow
\frac{N}{\langle N\cap \K(G)\rangle} \longrightarrow
 G^{\rm ab} \longrightarrow
 (G/N)^{\rm ab} \longrightarrow
 0
\]
with canonical maps is exact.
\end{lemma}

\begin{theorem}
\label{t:isoclasses}
The isoclinism classes of central CP extensions with factor group isomorphic to
$Q$ correspond to the orbits of the action of $\Aut Q$ on the subgroups of $\Bt
_0(Q)$ given by $(\varphi ,U)\mapsto \Bt_0(\varphi )U$, where $\varphi\in\Aut Q$
and $U\le\Bt_0(Q)$.
\end{theorem}

\begin{proof}
Let
\begin{equation*}
e: \,\, \xymatrix{ 1\ar[r] & N\ar[r]^{\chi} & G\ar[r]^{\pi} & Q\ar[r] & 1}
\end{equation*}
be a central CP extension. As $\chi (N)\cap K(G)=1$, it follows from
\cite{Bey82} and Lemma \ref{l:B0exact} that we have the following commutative
diagram with exact rows and columns:
\begin{equation*}
\xymatrix{
0 \ar[r] & \Bt _0(G)\ar[r]^{\Bt_0(\pi)} & \Bt_0(Q)\ar[r]^{\tilde{\theta}(e)} &
	N\ar[r]^\tau \ar@{=}[d] &
	G^{\rm ab} \ar[r]\ar@{=}[d] & Q^{\rm ab}\ar[r]\ar@{=}[d] & 0\\
 & \M (G)\ar[r]^{\M(\pi)}\ar@{->>}[u] & \M(Q)\ar[r]^{\theta _*(e)}\ar@{->>}[u] & 
 	N\ar[r]^\tau & G^{\rm ab} \ar[r] & Q^{\rm ab}\ar[r] & 0\\
 & \M_0(G)\ar@{^{(}->}[u]\ar[r] & \M_0(Q)\ar@{^{(}->}[u] & & & &
}
\end{equation*}
By the exactness we have that the image of $\chi\tilde{\theta}(e)$ is equal to
$\chi (N)\cap G'$ which equals to the image of $\chi\theta _*(e)$. Since $\chi$
is injective, it follows that $\tilde{\theta}(e)$ and $\theta _*(e)$ have the
same image. Furthermore, we claim that $\ker \tilde{\theta}(e)=\ker \theta _*(e)
/\M_0(Q)$. To this end, consider free presentations $G=F/R$, $N=S/R$, and
$Q=F/S$. Since the extension $e$ is CP, it follows that $\langle \K(F)\cap
S\rangle\le R$. With the above identifications we have that $\ker \theta
_*(e)=(F'\cap R)/[F,S]$ and $\ker\tilde{\theta}(e)=(F'\cap R)/\langle \K(F)\cap
S\rangle$.  As $\M_0(Q)=\langle \K(F)\cap S\rangle /[F,S]$, the equality
follows.

Let now 
\begin{equation*}
e_i: \,\, \xymatrix{ 1\ar[r] & N_i\ar[r]^{\chi _i} & G_i\ar[r]^{\pi _i} &
	Q_i\ar[r] & 1},\,\, (i=1,2)
\end{equation*}
by central CP extensions, and let $\eta :Q_1\to Q_2$ be an isomorphism of
groups. By \cite[Proposition III.2.3]{Bey82} we have that $\eta$ induces
isoclinism between $e_1$ and $e_2$ if and only if $\M(\eta)\ker \theta
_*(e_1)=\ker \theta _*(e_2)$. By the above, this is equivalent to
$\Bt_0(\eta)\ker\tilde{\theta}(e_1)=\ker\tilde{\theta}(e_2)$. The proof of
\cite[Proposition III.2.6]{Bey82} can now be suitably modified to obtain the
result, we skip the details.
\end{proof}

\section{Maximal CP extensions}
\label{s:covers}

In this section, we deal with studying maximal central CP extensions of a given
group. Maximal here refers to the size of the kernel in a suitable representative
extension under isoclinism. Recall that an extension
\[
\xymatrix{ 1\ar[r] & N\ar[r]^\chi & G\ar[r] & Q\ar[r] & 1}
\]
is termed to be {\em stem} whenever $\chi(N) \leq [G,G]$. The motivation comes
from the following lemma.

\begin{lemma} \label{l:isostemext}
Every central CP extension is isoclinic to a stem central CP extension.
\end{lemma}
\begin{proof}
The argument follows along the lines of \cite[Proposition III.2.6]{Bey82}. Let
\[ e: \,\, \xymatrix{ 1\ar[r] & N\ar[r] & G\ar[r] & Q\ar[r] & 1} \] be a central
CP extension. Put $U = \ker \tilde \theta (e)$, where $\tilde \theta(e)$ is the
homomorphism $\Bt_0(Q) \to N$ from the $5$-term exact sequence in Lemma
\ref{l:B0exact}. The subgroup $U$ of $\Bt_0(Q)$ determines a central CP
extension $\bar e$ of $\Bt_0(Q)/U$ by $Q$ via Theorem \ref{t:uct2} applied to
the epimorphism $\Bt_0(Q) \to \Bt_0(Q)/U$. Thus $\tilde \theta(\bar e)$
corresponds to the natural projection $\Bt_0(Q) \to \Bt_0(Q)/U$. Note that $\bar
e$ is a stem central CP extension isoclinic to $e$, cf. the proof of Theorem
\ref{t:isoclasses}. The kernel of the extension $\bar e$ is precisely
$\Bt_0(Q)/U \cong \im \tilde \theta (e) \cong \ker (N \to G/[G,G]) = N \cap
[G,G]$.
\end{proof}

Up to isoclinism of extensions, it therefore suffices to consider stem central
CP extensions.

Given a group $Q$, any stem central CP extension of a group $N$ by $Q$ with $|N|
= |\Bt_0(Q)|$ is called a {\em CP cover} of $Q$. The following theorem
justifies the terminology.

\begin{theorem} \label{t:cpcovers} Let $Q$ be a finite group given via a free
presentation $Q = F/R$. Set $H = F/\langle \K(F) \cap R \rangle$ and
$A = R/\langle K(F) \cap R \rangle$.
\begin{enumerate}
	\item $A$ is a finitely generated central subgroup of $H$ and its torsion 
	subgroup is $T(A) = ([F,F] \cap R)/\langle K(F) \cap R \rangle \cong \Bt_0(Q)$.
	\item Let $C$ be a complement to $T(A)$ in $A$. Then $H/C$ is a CP cover of $Q$.
	\item Let $G$ be a stem central CP extension of a group $N$ by $Q$. Then $G$
	is a homomorphic image of $H$ and in particular $N$ is a homomorphic image
	of $\Bt_0(Q)$.
	\item Let $G$ be a CP cover of $Q$ with kernel $N$. Then $N \cong \Bt_0(Q)$
	and $G$ is isomorphic to a quotient of $H$ by a complement of $T(A)$ in
	$A$.
	\item CP covers of $Q$ are precisely the stem central CP extensions of $Q$
	of maximal order.
	\item CP covers of $Q$ are represented by the cocycles
	$\tilde \varphi^{-1}(1_{\Bt_0(Q)})$ in $\HH^2(Q,\Bt_0(Q))$, where
	$\tilde \varphi$ is the mapping induced by the Universal Coefficients
	Theorem \ref{t:uct2}.
\end{enumerate}
\end{theorem}
\begin{proof}
This all follows from the arguments in \cite[Hauptsatz V.23.5]{Hup67} in
combination with the Hopf formula for the Bogomolov multiplier from
\cite{Mor12}.
\end{proof}

Using Theorem \ref{t:cpcovers}, a fast algorithm for computing the Schur
covering groups as developed in \cite{Nic93} may be combined with an algorithm
for determining $F/\langle \K(F) \cap R \rangle$ from \cite{Jez14} to
effectively determine the CP covers of a given group. It is straightforward to
combine the two implementations in {\sf GAP} \cite{GAP4}. 

\begin{corollary} \label{c:perfect}
The number of CP covers of a group $Q$ is at most $|\Ext(Q^{\rm ab},\Bt_0(Q))|$.
In particular, perfect groups have a unique CP cover.
\end{corollary}

\begin{example} 
Let $Q$ be a $4$- or $12$-cover of ${\rm PSL}(3,4)$. The group $Q$ is a quasi-simple group and it is shown in \cite{Kun08} that $\Bt_0(Q) \cong C_2$, so $Q$
has a unique proper CP cover.
\end{example}

We stress an important difference between Schur covering groups and CP covers,
indicating a more intimate connection of the latter with the theory of
(universal) covering spaces from algebraic topology \cite{Hat02}.

\begin{theorem} \label{t:b0cov}
The Bogomolov multiplier of a CP cover is trivial.
\end{theorem}
\begin{proof}
Let $G$ be a CP cover of $Q$ with kernel $N \cong \Bt_0(Q)$ satisfying
$N \leq Z(G) \cap [G,G]$ and $N \cap \K(G) = 1$.  Consider a CP cover 
$\xymatrix{H\ar[r]^\pi & G}$ with kernel $M \cong \Bt_0(G)$ satisfying
$M \leq Z(H) \cap [H,H]$ and $M \cap \K(H) = 1$.  The group $H$ is a central
extension of $L = \pi^{-1}(N)$ by $Q$, since $\pi$ preserves commutativity.
Moreover, we have $L \leq \pi^{-1}([G,G]) = [H,H]$ since $M \leq [H,H]$, and 
$L \cap \K(H) \leq \pi^{-1}(N \cap \K(G)) \cap \K(H) \leq M \cap \K(H) = 1$.  We
conclude that $H$ is a stem central CP extension of $L$ by $Q$, therefore
$|L| \leq |\Bt_0(Q)|$ by Theorem \ref{t:cpcovers}, and so $L \cong \Bt_0(Q)$.
This implies $M = \Bt_0(G) = 1$, as required.
\end{proof}

Note that a similar proof gives that the Bogomolov multiplier of a Schur
covering group is also trivial, see \cite[Lemma 2.4.1]{Fri91}.

For further use of Theorem \ref{t:b0cov}, we record a straightforward corollary
of Lemma \ref{l:b0cp}.

\begin{lemma}
\label{l:being-a-CP-cover}
Whenever $N$ is a central CP subgroup of a group $G$ with $\Bt_0(G) = 0$, then
$\Bt_0(G/N) \cong N \cap [G,G]$. If in addition $N \leq [G,G]$, then the group
$G$ is a CP cover of $G/N$ with kernel $N \cong \Bt_0(G/N)$.
\end{lemma}

It follows readily that central CP extensions behave much as topological
covering spaces.

\begin{corollary}
\label{c:covering-seq}
Let $Q$ be a group and $G$ a CP cover of $Q$. For every filtration of subgroups
$1 = N_0 \leq N_1 \leq \dots \leq N_n = \B_0(Q)$, there is a corresponding
sequence of groups $G_i = G/N_i$, where $G_i$ is a central CP extension of $G_j$
with kernel $N_j/N_i \cong \B_0(G_j)/\B_0(G_i)$  whenever $i \leq j$.
\end{corollary}

We now explore CP covers with respect to isoclinism. At first we list some auxiliary results.

\begin{lemma}
\label{l:cpequiv}
Let
\begin{equation*}
e: \,\, \xymatrix{ 1\ar[r] & N\ar[r]^{\chi} & G\ar[r]^{\pi} & Q\ar[r] & 1}
\end{equation*}
be a central CP extension. Then $\pi(Z(G))=Z(Q)$ and $Z(G) \cong  N \times Z(Q)$.
\end{lemma}

\begin{proof}
It is straightforward to see that if
\[
e_i: \,\, \xymatrix{ 1\ar[r] & N\ar[r]^{\chi _i} & G_i\ar[r]^{\pi _i} & Q\ar[r] & 1}
\]
are equivalent
central extensions for $i=1,2$, then $\pi_1(Z(G_1))=\pi_2(Z(G_2))$. Thus we may replace the
extension $e$ by the extension
\begin{equation*}
\xymatrix{ 1\ar[r] & N\ar[r] & G[\omega]\ar[r]^\epsilon & Q\ar[r] & 1},
\end{equation*}
that is obtained similarly as in the proof of Proposition \ref{p:classext}. As
$\omega\in\ZCP^2(Q,N)$, the condition that $(n,q)\in Z(G[\omega])$ is equivalent
to $q\in Z(Q)$. Therefore $\epsilon (Z(G[\omega]))=Z(Q)$.
\end{proof}

\begin{lemma} \label{l:cpcovcenterstem}
Let $G$ be a CP cover of $Q$. Then $Z(G) \cong Z(Q) \times \Bt_0(Q)$, and $G$ is
stem if and only if $Q$ is stem.
\end{lemma}
\begin{proof}
The first part follows from Lemma \ref{l:cpequiv}. The second part then follows
from the first one and the fact that $\Bt_0(Q) \leq [G,G]$.
\end{proof}

It follows from the latter lemma that the central quotient of a CP cover is
naturally isomorphic to the central quotient of the base group, and so the
nilpotency class of a CP cover does not exceed that of the base group. This is
all a special case of the following observation.

\begin{proposition} \label{p:coversareisocl}
CP covers of isoclinic groups are isoclinic.
\end{proposition}
\begin{proof}
Let $G_1$ be a CP cover of a group $Q_1$ with the covering projection $p_1
\colon G_1 \to Q_1$ and let $Q_2$ be isoclinic to $Q_1$ via the compatible pair
of isomorphisms $\alpha \colon Q_2/Z(Q_2) \to Q_1/Z(Q_1)$ and 
$\beta \colon [Q_2, Q_2] \to [Q_1,Q_1]$.  Let $G_2$ be a CP cover of $Q_2$ with
the covering projection $p_2 \colon G_2 \to Q_2$. We show that $G_2$ is
isoclinic to $G_1$. To this end, let $\bar p_i \colon G_i/Z(G_i) \to Q_i/Z(Q_i)$
be the natural homomorphisms induced by $p_i$'s. Lemma \ref{l:isostemext} implies
that $\bar p_i$ is in fact an isomorphism. Define
$\tilde \alpha \colon G_2/Z(G_2) \to G_1/Z(G_1)$ as
$\tilde \alpha = (\bar p_1)^{-1} \alpha \bar p_2$. This is clearly an
isomorphism. Next, observe that Theorem \ref{t:cpcovers} shows that the covering
projections $p_i$ also induce isomorphisms 
$p_i \curlywedge p_i \colon [G_i, G_i] \to Q_i \curlywedge Q_i$ defined as
$[x,y] \mapsto p_i(x) \curlywedge p_i(y)$. Furthermore, it is shown in
\cite{Mor13} that $\alpha$ induces an isomorphism
$\alpha^{\curlywedge} \colon Q_2 \curlywedge Q_2 \to Q_1 \curlywedge Q_1$ via
$\alpha^{\curlywedge}(x_1 \curlywedge x_2) =  y_1 \curlywedge y_2$, where
$y_iZ(Q_1) = \alpha(x_i Z(Q_2))$. Now define 
$\tilde \beta \colon [G_2, G_2] \to [G_1, G_1]$ as
$\tilde \beta = (p_2 \curlywedge p_2) \alpha^{\curlywedge} (p_1 \curlywedge p_1)^{-1}$. 
This is clearly an isomorphism, and it readily follows from the compatibility
relations between $\alpha$ and $\beta$ that the isomorphisms $\tilde \alpha$ and
$\tilde \beta$ are also compatible. These induce an isoclinism between the CP
covers $G_1$ and $G_2$.
\end{proof}

As a corollary, the derived subgroup of a CP cover is uniquely determined. Note
that given a group $Q$ and its CP cover $G$, we have 
$[G,G] \cong Q \curlywedge Q$ by Theorem \ref{t:cpcovers}. In particular, groups
belonging to the same isoclinism family have naturally isomorphic curly exterior
squares, and therefore also Bogomolov multipliers.

Let $\Phi$ be an isoclinism family of finite groups, referred to as the {\em
base family}, and let $G$ be an arbitrary group in $\Phi$. By Proposition
\ref{p:coversareisocl}, CP covers of $G$ all belong to the same isoclinism
family. We denote this family by $\tilde \Phi$ and call it the {\em covering
family} of $\Phi$.

\begin{proposition} \label{p:covisoclin}
Every group in a covering family is a CP cover of a group in the base family.
\end{proposition}
\begin{proof}

Let $G_1$ be a CP cover of a group $Q_1$ with the covering projection 
$p_1 \colon G_1 \to Q_1$ and let $G_2$ be isoclinic to $G_1$ via the compatible
pair of isomorphisms $\alpha \colon G_2/Z(G_2) \to G_1/Z(G_1)$ and 
$\beta \colon [G_2, G_2] \to [G_1,G_1]$.  By Theorem \ref{t:b0cov}, we have
$\Bt_0(G_1) = 0$, and so $\Bt_0(G_2) = 0$ by Proposition \ref{p:coversareisocl}.
The commutator homomorphism $\kappa_i \colon G_i \curlywedge G_i \to [G_i, G_i]$
is therefore an isomorphism, and we implicitly identify the two groups. Consider
the group $N = \beta^{-1} \Bt_0(Q_1) \leq [G_2, G_2]$. Note that $N$ is central
in $G_2$. Furthermore, whenever $[x_1, x_2] \in N$ for some $x_1, x_2 \in G_2$
with $x_i Z(G_2) = y_i Z(G_1)$, we have 
$[y_1, y_2] = \beta ([x_1, x_2]) \in \Bt_0(Q_1)$, and so 
$[x_1, x_2] = \beta^{-1}([y_1, y_2]) = 1$ since the covering projection
$G_1 \to Q_1$ is commutativity preserving.  Now put $Q_2 = G_2/N$. By Lemma
\ref{l:being-a-CP-cover}, the group $G_2$ is a CP cover of $Q_2$ with kernel
$N \cong \Bt_0(Q_2)$. Finally, it is straightforward that the isomorphisms
$\alpha$ and $\beta$ naturally induce an isoclinism between the groups 
$Q_2 = G_2/\beta^{-1}(\Bt_0(Q_1))$ and $G_1/\Bt_0(Q_1) \cong Q_1$.
\end{proof}

Note that Lemma \ref{l:cpcovcenterstem} now implies that CP covers of the stem
of the base family form the stem of the covering family.

The following examples show that a given isoclinism family can be a covering
family for more than one base family. Moreover, a group in a covering family can
be a CP cover of non-isomorphic groups belonging to the same base family.

\begin{example} \label{e:64} 
Consider the isoclinism family that contains groups of smallest possible order
having non-trivial Bogomolov multipliers \cite{Chu09}. This is the family
$\Phi_{16}$ of \cite{Jam90}.  Its stem groups are of order $64$, and its
covering family $\tilde \Phi_{16}$ is precisely the isoclinism family
$\Phi_{36}$ of \cite{Jam90}, whose stem groups are of order $128$.
\end{example}

\begin{example}
Let $G$ be a Schur covering group of the abelian group $C_4^4$, generated by
$g_1, g_2, g_3, g_4$. Put $w = [g_1, g_2][g_3, g_4]$ and set 
$G_1 = G/\langle w \rangle$, $G_2 = G/\langle w^2 \rangle$. It is readily
verified that neither $w$ nor $w^2$ is a commutator in $G$. Since 
$\Bt_0(G) = 0$, if follows that $G$ is a CP cover of both $G_1$ and of $G_2$.
Applying Lemma \ref{l:B0exact} gives $\Bt_0(G_1) \cong C_4$ and 
$\Bt_0(G_2) \cong C_2$, so $G_1$ and $G_2$ do not belong to the same isoclinism
family.
\end{example}

\begin{example}
Let $Q$ be a stem group in the family $\Phi_{30}$ of \cite{Jam90} and let $G$ be
a CP cover of $Q$. It is shown in \cite{Jez14} that
$\Bt_0(Q) = \langle w_1 \rangle \times \langle w_2 \rangle \cong C_2 \times C_2$
for some $w_1, w_2 \in G$. Set $G_1 = G/\langle  w_1 \rangle$ and 
$G_2 = G/\langle  w_2 \rangle$. The groups $G_1$ and $G_2$ are isoclinic and non-isomorphic 
groups of order $256$, and $G$ is a CP cover of both of them. Is can
be verified using the algorithm for computing CP covers that the groups $G_1,
G_2$ in fact have exactly two non-isomorphic CP covers in common.
\end{example}

It is well-known that Schur covering groups of a given group are all isoclinic,
see for example \cite[Satz V.23.6]{Hup67}. Neither Proposition
\ref{p:coversareisocl} nor Proposition \ref{p:covisoclin}, however, has a
counterpart in the theory of Schur covering groups, as already the following
simple example shows.

\begin{example}
Let $\Phi$ be the isoclinism family of all finite abelian groups. We plainly
have $\tilde \Phi = \Phi$. Let $p$ be an arbitrary prime.  The Schur cover of
$C_{p^2}$ is $C_{p^2}$, and the Schur cover of $C_p \times C_p$ is isomorphic to
the unitriangular group ${\rm UT}_3(p)$. The two covers are not isoclinic. Note
also that the group $C_p \times C_p$ is not a Schur covering group of any group.
\end{example}

\section{Minimal CP extensions}
\label{s:appminimal}

In this section, we focus on central CP extensions of a cyclic group of prime
order by some given group $Q$. We call such extensions {\em minimal} CP
extensions. By Corollary \ref{c:covering-seq}, every central CP extension is
built from a sequence of such minimal extensions. As in the classical theory of
central extensions, this corresponds to considering $\FF_p$-cohomology. We
thus set $\HHCP^2(Q) = \HHCP^2(Q, \FF_p)$, the action of $Q$ on $\FF_p$ being
trivial. Relying on Theorem \ref{t:cpcovers}, the heart of the matter here is
relating a given presentation of $Q$ with the object $\HHCP^2(Q)$. The following
result is obtained.

\begin{theorem} \label{t:h2cprank}
The group $\HHCP^2(Q)$ is elementary abelian of rank $\dd(Q) + \dd(\B_0(Q))$.
\end{theorem}
\begin{proof}

Let $Q = F/R$ be a presentation of $Q$. Consider first the canonical central CP
extension $H = F/ \langle \K(F) \cap R \rangle$ of $Q$. The kernel of this
extension is the group $A = R/ \langle \K(F) \cap R \rangle$. 

We first claim that $\HHCP^2(H) = 0$. By Lemma \ref{l:B0exact}, we have $\B_0(H)
= 0$, and it then follows from Theorem \ref{t:uct2} that $\HHCP^2(H) =
\Ext(H^{\rm ab}, \FF_p) = 0$.

Next we show that the minimal CP extensions are precisely the kernel of the
inflation map from $Q$ to $H$:
\[
\textstyle \HHCP^2(Q) = \ker \left(\inf_Q^H \colon \HH^2(Q) \to \HH^2(H)  \right).
\]
Indeed, it follows from the above claim that $\HHCP^2(Q) \leq \ker \inf_Q^H$.
Conversely, let $\omega \in \ker \inf_Q^H$. Hence there is a function
$\phi \colon H \to \FF_p$ such that $\inf_Q^H(\omega)(x_1, x_2) = \phi(x_1) + \phi(x_2) - \phi(x_1 x_2)$. Pick any commuting pair $u,v \in Q$. Then there exists a commuting lift $\tilde u, \tilde v \in H$ of these elements. Therefore
$\omega(u, v) = \inf_Q^H(\omega)(\tilde u,\tilde v) = \inf_Q^H(\omega)(\tilde v,\tilde u)
= \omega(v, u)$, and so $\omega \in \HHCP^2(Q)$.

Let us now restrict to choosing the presentation $Q = F/R$ to be minimal in the
sense that $\dd(Q) = \dd(F)$. In this case, we invoke the inflation-restriction
cohomological exact sequence for the surjection $H \to Q$ with kernel $A$.
Together with the above, it immediately follows that $\HHCP^2(Q) \cong \Hom(A,
\FF_p)$. Finally, we have by Theorem \ref{t:cpcovers} that the torsion $T(A)
\cong \B_0(Q)$ in $A$ has a free complement of rank $\dd(F) = \dd(Q)$. The proof
is complete.
\end{proof}

We expose a corollary of the above proof.

\begin{corollary} \label{c:boundingrankwithpresentation}

Let $Q = F/R$ be a presentation with $\dd(Q) = \dd(F)$.  Let $\rr(F,R)$ be the
minimal number of relators in $R$ that generate $R$ as a normal subgroup of $F$,
and let $\rr_{\K}(F,R)$ be the number of relators among these that belong to
$\K(F)$. Then $\dd(\B_0(Q)) \leq \rr(F,R) - \rr_{\K}(F,R) - \dd(Q)$.

\end{corollary}
\begin{proof}
Going back to the proof of Theorem \ref{t:h2cprank}, it is clear that $\rank A
\leq \rr(F,R) - \rr_{\K}(F,R)$. The claim follows immediately.
\end{proof}

The corollary may be applied to show that the Bogomolov multiplier of a group is
trivial. This works with classes of groups which may be given by a presentation
with many simple commutators among relators. As an example, the group of
unitriangular matrices $\UT_n(p)$ has a presentation in which all relators are
commutators (see \cite{Bis01}), whence immediately $\B_0(\UT_n(p)) = 0$. The
same holds for lower central quotients of $\UT_n(p)$. 
This was already proved in \cite{Mic13}, see also \cite{Jez14a}.
Another example is the
braid group $B_n$ with $n-1$ generators and $n-2$ braid relators that are not
commutators, thereby again $\B_0(B_n) = 0$.

\section{Bounds for $\B_0$}
\label{s:appbounds}

Using the theory of CP covers, we now show how one can produce bounds on the
number of isoclinism classes of central CP extensions in terms of the internal
structure  of the given group. Equivalently, we bound the size of the Bogomolov
multiplier. The first result is an adaptation of the argument from \cite{Jon74}.

\begin{proposition}
\label{p:sizeofbog}

Let $Q$ be a finite group and $S$ a normal subgroup such that $Q/S$ is cyclic.
Then $|\B_0(Q)|$ divides $|\B_0(S)| \cdot |S^{\rm ab}|$, and $\dd(\B_0(Q)) \leq 
\dd(\B_0(S)) + \dd(S^{\rm ab})$.

\end{proposition}
\begin{proof}
Let $G$ be a CP cover of $Q$. Thus $G$ contains a subgroup $N \leq [G,G] \cap
Z(G)$ such that $G/N \cong Q$ and $N \cong \B_0(Q)$. Choose $X$ in $G$ such that
$X/N \cong S$. We may write $G = \langle u, X \rangle$ for some $u$. There is
thus an epimorphism $\theta \colon X \to [G,G]/[X,X]$ given by $\theta(x) =
[u,x][X,X]$. Therefore $|\B_0(Q)| = |N| = |N/(N \cap [X,X])| \cdot |N \cap
[X,X]|$. Now, since $NX' \leq \ker \theta$, it follows that $|N/(N \cap [X,X])|
\leq |[G,G]/[X,X]| \leq |X / N [X,X]| = |S^{\rm ab}|$. Observe that the CP covering
extension $G$ of $Q$ induces a central CP extension $X$ of $S$ with kernel $N$.
Whence by Lemma \ref{l:isostemext}, we have that $N \cap [X,X]$ is the kernel of the
associated stem extension. It now follows from Theorem \ref{t:cpcovers} that $|N \cap
[X,X]| \leq |\B_0(Q)|$. This completes the proof of the first claim. For the
second one, we similarly have $\dd(\B_0(Q)) = \dd(N) \leq \dd(N/(N \cap [X,X]))
+ \dd(N \cap [X,X])$. The result follows from $\dd(N / (N \cap [X,X])) \leq
\dd([G,G]/[X,X]) \leq \dd(S^{\rm ab})$.
\end{proof}

Next, we also provide a bound for the exponent. This is an analogy of
\cite{Jon73}.

\begin{proposition}
\label{p:expofbog}

Let $Q$ be a finite group and $S$ a subgroup. Then $\B_0(Q)^{|Q:S|}$ embeds into $\B_0(S)$.

\end{proposition}
\begin{proof}
Let $G$ be a CP cover of $Q$. Again, $G$ contains a subgroup $N \leq [G,G] \cap
Z(G)$ such that $G/N \cong Q$ and $N \cong \B_0(Q)$. Choose $X$ in $G$ such that
$X/N \cong S$. Consider the transfer map $\theta \colon G \to X/[X,X]$. Since
$N$ is central in $G$, we have $\theta(n) = n^{|Q:S|} [X,X]$ for all $n \in N$.
But as $N \leq [G,G]$, we must also have that $N \leq \ker \theta$. Therefore
$N^{|Q:S|} \leq N \cap [X,X]$. As in the proof of the previous proposition, we
have that $N \cap [X,X]$ embeds into $\B_0(S)$. This completes the proof.
\end{proof}

These results may be applied in various ways, depending on the structural
properties of the group in question, to provide some absolute bounds on the
order, rank or exponent of the Bogomolov multiplier. As an example, consider a
$p$-group $Q$ that has a maximal subgroup $M$ with $\B_0(M) = 0$. The above
propositions imply that for such groups, $\B_0(Q)$ is elementary abelian of rank
at most $\dd(M)$. Such groups include $\B_0$-minimal groups, the building blocks
of groups with non-trivial Bogomolov multipliers, and were inspected to some
extent in \cite{Jez14}. It was shown that every $\B_0$-minimal group can be
generated by at most $4$ elements. Schreier's index formula $\dd(M) - 1 \leq
|Q:M| (\dd(Q) - 1)$, cf. \cite[6.1.8]{Rob1}, then gives an absolute upper bound on the number of
generators of a maximal subgroup $M$. Whence the following corollary.

\begin{corollary}
The Bogomolov multiplier of a $\B_0$-minimal $p$-group is an elementary abelian
group of rank at most $3p + 1$.
\end{corollary}

Another direct application is to consider any abelian subgroup $A$ of 
a given group $Q$. Since $\B_0(A) = 0$, we have the following.

\begin{corollary}
Let $Q$ be a finite group and $A$ an abelian subgroup. Then $\exp \B_0(Q)$
divides $|Q:A|$.
\end{corollary}

\section{Commuting probability}
\label{s:comm-prob}

A probabilistic approach to the study of Bogomolov multipliers has been
undertaken in \cite{Jez15}, where the impact of the commuting probability on the
Bogomolov multiplier was explored. It turns out that commutativity preserving
extensions provide a natural setting for both commuting probability and
Bogomolov multipliers.  This is based on the following observation.

\begin{proposition}
\label{p:comm-prob}
An extension $ \xymatrix{N\ar[r] & G\ar[r]^{\pi} & Q} $ is a central CP
extension if and only if $\cp(G) = \cp(Q)$.
\end{proposition}
\begin{proof}
Observe the homomorphism $\pi^2 \colon G \times G \to Q \times Q$. Note that
commuting pairs in $G$ map to commuting pairs in $Q$,  hence
\begin{equation} \label{eq:preimg-of-comm} 
(\pi^2)^{-1}(\{ (x,y) \in Q \times Q \mid [x,y] = 1 \}) \supseteq
\{ (x,y) \in G \times G \mid [x,y] = 1 \}.
\end{equation} 
The containment \eqref{eq:preimg-of-comm} is an equality if and only if the
extension is CP and $N$ is a central subgroup of $G$. On the other hand, notice
that the fibres of $\pi^2$ are of order $|N|^2$, therefore 
$\cp(G) = |\{ (x,y) \in G \times G \mid [x,y] = 1 \}|/|G|^2 \leq
|N|^2|\{ (x,y) \in Q \times Q \mid [x,y] = 1 \}|/|G|^2 = \cp(Q)$ 
with equality precisely when \eqref{eq:preimg-of-comm} is an equality. This
completes the proof.
\end{proof}

\begin{remark}
Consider a central extension
$ \xymatrix{\langle z \rangle\ar[r] & G\ar[r]^{\pi} & Q} $.
It follows from the above proof that this extension
is a CP extension if and only if all conjugacy classes of $Q$ lift with respect
to $\pi$ to exactly $p$ different conjugacy classes in $G$.
\end{remark}

The study of central CP extensions is thus equivalent to the study of extensions
which preserve commuting probability. This may be exploited in providing a
connection between the Bogomolov multiplier and commuting probability based on
CP extensions. We give a simple example illustrating this.

\begin{corollary}
For every number $p$ in the range of the commuting probability function, there
exists a group $G$ with $\cp(G) = p$ and  $\B_0(G) = 0$.
\end{corollary}
\begin{proof}
Let $Q$ be an arbitrary group with $\cp(Q) = p$, and let $G$ be a CP-cover of
$Q$. Then $\cp(G) = p$ by Proposition \ref{p:comm-prob} and $\B_0(G) = 0$ by
Theorem \ref{t:b0cov}.
\end{proof}

Another way to look at this relation is on the level of isoclinism families.
As a direct consequence of Corollary \ref{c:covering-seq}, we have that
for every isoclinism family $\Phi$  and every subgroup $N$ of $\B_0(\Phi)$,
there is a family $\Phi'$ with  $\cp(\Phi') = \cp(\Phi)$ and $\B_0(\Phi') = N$.

\begin{example}
Observe the isoclinism family $\Phi_{16}$ as given in Example \ref{e:64}. We
have $\cp(\Phi_{16}) = \cp(\Phi_{36}) = 1/4$, while $\B_0(\Phi_{16}) \cong C_2$
and $\B_0(\Phi_{36}) = 0$.
\end{example}

This connection also sheds new light on the results of \cite{Jez15}. There, we
have observed the structure of the Bogomolov multiplier while fixing a large
commuting probability. First of all, those results can be applied in the context
of CP extensions.

\begin{corollary} \label{c:14}
Let $Q$ be a finite group with $\cp(Q) > 1/4$. Then every central CP extension
of $Q$ is isoclinic to an extension with a trivial kernel.
\end{corollary}
\begin{proof}
The Bogomolov multiplier of $Q$ is trivial by \cite[Corollary 1.2]{Jez15}. Every
central CP extension of $Q$ is isoclinic to a stem extension by Lemma
\ref{l:isostemext}, and the kernel of the latter extension must be trivial by
Theorem \ref{t:cpcovers}.
\end{proof}

And secondly, the bounds for the Bogomolov multiplier from Section
\ref{s:appbounds} can be applied in the setting of commuting probability.  This
is, in a way, a non-absolute version of the main result of \cite{Jez15}.

\begin{theorem}
\label{t:boguniversalbound}

Let $\epsilon > 0$, and let $Q$ be a group with $\cp(Q) > \epsilon$. Then
$|\B_0(Q)|$ can be bounded in terms of a function of $\epsilon$ and $\max \{
\dd(S) \mid S \textrm{ a Sylow subgroup of $Q$} \}$. Moreover, $\exp \B_0(Q)$
can be bounded in terms of a function of $\epsilon$.

\end{theorem}
\begin{proof}

Since the $p$-part of $\B_0(Q)$ embeds into the Bogomolov multiplier of a
$p$-Sylow subgroup of $Q$, we are immediately reduced to considering only
$p$-groups. It follows from \cite{Neu89,Ebe15} that $Q$ has a subgroup $K$ of
nilpotency class $2$ with $|Q:K|$ and $|[K,K]|$ both bounded by a function of
$\epsilon$. Applying Proposition \ref{p:sizeofbog} repeatedly on a sequence of
subgroups from $Q$ to $K$, each of index $p$ in the previous one, it follows
that $\dd(\B_0(Q))$ can be bounded in terms of $\epsilon$ and $\dd(\B_0(K))$.
Now, $\dd(\B_0(K)) \leq \dd(M(K))$, and we can use the Ganea map $[K,K] \otimes
K/[K,K] \to M(K)$ whose cokernel embeds into $M(K/[K,K])$. Note that $\dd([K,K]
\otimes K/[K,K]) \leq \dd(K)^2$ and $\dd(M(K/[K,K])) \leq \binom{\dd(K)}{2}$.
Whence we obtain a bound for $\dd(\B_0(Q))$ in terms of $\epsilon$ and $\dd(Q)$.
For the exponent, use Proposition \ref{p:expofbog} to bound $\exp \B_0(Q)$ by a
function of $|Q:K|$ and $\exp \B_0(K)$. If $K$ is abelian, then we are done. If
not, then choose a commutator $z$ in $K$. Set $J_z = \langle x \curlywedge y \mid [x,y] = z \rangle \leq Q \curlywedge Q$, and denote by $X$ the kernel of the map $\B_0(K)\to\B_0(K/\langle z \rangle)$. Then it follows from \cite{Jez15} that there is a commutative diagram as
follows. 

\begin{equation*}
\xymatrix{
& 1 \ar[d] & 1 \ar[d]& 1 \ar[d]& \\
1 \ar[r] & X\ar[d] \ar[r] & \B_0(K)\ar[d] \ar[r] & \B_0(K/\langle z \rangle)\ar[d] \ar[r] & 1 \\
1 \ar[r] & J_z \ar[d]\ar[r] & K \curlywedge K \ar[d]\ar[r] & K/\langle z \rangle \curlywedge K/\langle z \rangle \ar[d]\ar[r] & 1 \\
1 \ar[r] & \langle z \rangle\ar[d] \ar[r] & [K,K]\ar[d] \ar[r] & [K,K]/\langle z \rangle \ar[d]\ar[r] & 1 \\
& 1  & 1 & 1 &
}
\end{equation*}

Observe that $\exp J_z = p$, and so $\exp X = p$. It then follows that $\exp
\B_0(K)$ is at most $p \cdot \exp \B_0(K / \langle z \rangle)$. Repeating this
process with $K / \langle z \rangle$ instead of $z$ until we reach an abelian
group, we conclude that $\exp \B_0(K)$ divides $|[K,K]|$. The latter is bounded
in terms of $\epsilon$ alone. The proof is now complete.
\end{proof}

We end with an intriguing corollary concerning the exponent of the Schur
multiplier.

\begin{corollary}

Given $\epsilon > 0$, there exists a constant $C = C(\epsilon)$ such that for
every group $Q$ with $\cp(Q) > \epsilon$, we have $\exp M(Q) \leq C \cdot \exp Q$.

\end{corollary}
\begin{proof}
We have that $\exp M(Q) \leq \exp \B_0(Q) \cdot \exp M_0(Q)$ and $\exp M_0(Q) \leq \exp Q$. Now apply Theorem \ref{t:boguniversalbound}.
\end{proof}

\end{document}